\documentclass[12pt,draftcls,onecolumn]{IEEEtran}

\usepackage{amsmath,amssymb,pstricks,color,dsfont}
\usepackage{graphicx}
\usepackage{algorithm,algorithmic,xspace}

\newcommand{\nnum}{\nonumber}

\newcommand{\EQ}{\begin{eqnarray}}
\newcommand{\EN}{\end{eqnarray}}
\newcommand{\EQQ}{\begin{eqnarray*}}
\newcommand{\ENN}{\end{eqnarray*}}

\newcommand{\bremark}{\begin{remark} \begin{rm} }
\newcommand{\eremark}{ \end{rm} \rule{1mm}{2mm}
\end{remark} }
\newcommand{\btheorem}{\begin{theorem} \begin{rm} }
\newcommand{\etheorem}{ \end{rm} \rule{1mm}{2mm}
\end{theorem} }
\newcommand{\blemma}{\begin{lemma} \begin{rm} }
\newcommand{\elemma}{ \end{rm} \rule{1mm}{2mm}
\end{lemma} }
\newcommand{\bcorollary}{\begin{corollary} \begin{rm} }
\newcommand{\ecorollary}{ \end{rm} \rule{1mm}{2mm}
\end{corollary} }
\newcommand{\bdefinition}{\begin{definition}\begin{rm} }
\newcommand{\edefinition}{ \end{rm} \rule{1mm}{2mm}
\end{definition} }
\newcommand{\bproposition}{\begin{proposition} \begin{rm} }
\newcommand{\eproposition}{ \end{rm} \rule{1mm}{2mm}
\end{proposition} }
\newcommand{\bexample}{\begin{example} \begin{rm} }
\newcommand{\eexample}{ \end{rm} \rule{1mm}{2mm}
\end{example} }
\newcommand{\basm}{\begin{assumption} \begin{rm}}
\newcommand{\easm}{\end{rm} 
\end{assumption}}

\newcommand{\real}{\mathds{R}}

\newcommand{\inputseq}{\mathbf{u}}

\newcommand{\stateseq}{\mathbf{x}}

\newcommand{\subscr}[2]{#1_{\textup{#2}}}

\newtheorem{theorem}{\bf Theorem}[section]
\newtheorem{lemma}{\bf Lemma}[section]
\newtheorem{definition}{\bf Definition}[section]
\newtheorem{remark}{\bf Remark}[section]
\newtheorem{corollary}{\bf Corollary}[section]
\newtheorem{proposition}{\bf Proposition}[section]
\newtheorem{example}{\bf Example}[section]
\newtheorem{assumption}{\bf Assumption}[section]

\newcommand\oprocendsymbol{\hbox{$\bullet$}}
\newcommand\oprocend{\relax\ifmmode\else\unskip\hfill\fi\oprocendsymbol}

\date{}

\begin{document}

\title{On the performance analysis of resilient networked control systems under replay attacks}

\author{ Minghui Zhu and Sonia Mart{\'\i}nez \thanks{M. Zhu is with the
    Laboratory for Information and Decision Systems,
    Massachusetts Institute of Technology, 77 Massachusetts Avenue, Cambridge MA, 02139, {\tt\small (mhzhu@mit.edu)}. S. Mart{\'\i}nez is with the Department of
    Mechanical and Aerospace Engineering, University of California,
    San Diego, 9500 Gilman Dr, La Jolla CA, 92093, {\tt\small
      (soniamd@ucsd.edu)}.\newline This work was supported by AFOSR
    Grant 11RSL548.}}

\maketitle

\begin{abstract}
  This paper studies a resilient control problem for discrete-time,
  linear time-invariant systems subject to state and input
  constraints. State measurements and control commands are transmitted
  over a communication network and could be corrupted by
  adversaries. In particular, we consider the replay attackers who
  maliciously repeat the messages sent from the operator to the
  actuator. We propose a variation of the receding-horizon control law
  to deal with the replay attacks and analyze the resulting
  system performance degradation. A class of competitive (resp. cooperative) resource allocation problems for resilient networked control systems is also investigated.
\end{abstract}

\section{Introduction}

The recent advances of information technologies have boosted the
emergence of networked control systems where information networks are
tightly coupled to physical processes and human intervention. Such
sophisticated systems create a wealth of new opportunities at the
expense of increased complexity and system vulnerability. In
particular, malicious attacks in the cyber world are a current
practice and a major concern for the deployment of networked control
systems. Thus, the ability to analyze their consequences becomes of
prime importance in order to enhance the resilience of these
new-generation control systems.

This paper considers a single-loop remotely-controlled system, in
which the plant, together with a sensor and an actuator, and the
system operator are spatially distributed and connected via a
communication network. In particular, state measurements are
communicated from the sensor to the system operator through the
network; then, the generated control commands are transmitted to the
actuator through the same network. This model is an abstraction of a
variety of existing networked control systems, including supervisory
control and data acquisition (SCADA) networks in critical
infrastructures (e.g., power systems and water management systems) and
remotely piloted unmanned aerial vehicles (UAVs). The objective of the
paper is to design and analyze resilient controllers against replay attacks.

\emph{Literature review.} Recently, the cyber security of control systems has received increasing attention. The research effort has been devoted to studying two aspects: attack detection and attack-resilient control. Regarding attack detection, a particular class of cyber attacks, namely \emph{false data injection}, against state estimation is studied in~\cite{FP-RC-FB:11,AT-SA-HS-KHJ-SSS:10,LX-YM-BS:10}. The
paper~\cite{YM-BS:09} studies the detection of the \emph{replay
attacks}, which maliciously repeat transmitted data. In the context of multi-agent systems, the papers of~\cite{FP-AB-FB:09b,SS-CNH:11}
determine conditions under which consensus multi-agent systems can
detect misbehaving agents. As for attack-resilient control, the papers~\cite{SA-XL-SS-AMB:10,MZ-SM:11b,MZ-SM:ACC11} are
devoted to studying \emph{deception attacks}, where attackers
intentionally modify measurements and control
commands. \emph{Denial-of-service} (DoS) attacks destroy the data
availability in control systems and are tackled in recent
papers~\cite{SA-AC-SSS:09,SA-GAS-SSS:10,GKB-VG-PJA:11,AG-CL-TB:10}. More
specifically, the papers~\cite{SA-AC-SSS:09,AG-CL-TB:10} formulate
finite-horizon LQG control problems as dynamic zero-sum games between
the controller and the jammer. In~\cite{SA-GAS-SSS:10}, the authors
investigate the security independency in infinite-horizon LQG against
DoS attacks, and fully characterize the equilibrium of the induced
game. In our paper~\cite{MZ-SM:acc-12},
a distributed receding-horizon control law is proposed to ensure that
vehicles reach the desired formation despite the DoS and replay
attacks.

The problems of control and estimation over unreliable communication
channels have received considerable attention over the last
decade~\cite{JH-PN-YX:07}. Key issues include band-limited
channels~\cite{DL-JPH:05,GNN-FF-SZ-RJE:07},
quantization~\cite{RWB-DL:00,GNN-RJE-IMYM-WM:04}, packet
dropout~\cite{VG-NM:10,OCI-SY-TB:06,LS-BS-MF-KP-SSS:07},
delay~\cite{MSB-SMP-WZ:00} and
sampling~\cite{DN-AT:04}. Receding-horizon networked control is
studied in~\cite{BD:11,VG-BS-SA-AG:06,DP-PDC:08} for package dropouts
and in~\cite{KK-KH:12,GPL-JXM-DR-SCC:06} for transmission
delays. Package dropouts and DoS attacks (resp. transmission delays and replay attacks) cause similar affects to control systems. So the existing receding-horizon control approaches exhibit the robustness to certain classes of DoS and replay attacks under their respective assumptions. However, none of these papers characterizes the performance
degradation of receding-horizon control induced by the communication
unreliability.

\emph{Contributions.} We study a variation of the receding-horizon
control under the replay attacks. A set of sufficient
conditions are provided to ensure asymptotical and exponential
stability. More importantly, we derive a simple and explicit relation
between the infinite-horizon cost and the computing and attacking
horizons. By using such relation, we characterize a class of competitive (resp. cooperative) resource allocation problems for resilient networked control systems as convex games (resp. programs). The preliminary results are published in~\cite{MZ-SM:ACC11}
where receding-horizon control is used to deal with a class of
deception attacks. The technical relations between this paper
and~\cite{MZ-SM:ACC11} will be explained at the very beginning of
Section~\ref{sec:analysis}.

\section{Attack-resilient receding-horizon control}

\subsection{Description of the controlled system}

Consider the following discrete-time, linear time-invariant dynamic system:
\begin{align}
  x(k+1) = A x(k) + B u(k),\label{e10}
\end{align}
where $x(k)\in\real^n$ is the system state, and $u(k)\in\real^m$ is
the system input at time $k\ge 0$. The matrices $A\in \real^{n\times
  n}$ and $B \in \real^{n\times m}$ represent the state and the input
matrix, respectively. States and inputs of system~\eqref{e10} are constrained
to be in some sets; i.e., $x(k)\in X$ and $u(k)\in U$, for all
$k\geq0$, where $0\in X\subseteq \real^n$ and $0\in U\subseteq
\real^m$. The quantities $\|x(k)\|_P^2$ and $\|u(k)\|_Q^2$ are running
state and input costs, respectively, for some $P$ and $Q$
positive-definite and symmetric matrices. We assume the following holds for the
system:

\begin{assumption}\textbf{(Stabilizability)} The pair $(A,B)$ is stabilizable. \oprocend\label{asm2}
\end{assumption}

This assumption ensures the existence of $K$ such that the spectrum
$\sigma(\bar{A})$ is strictly inside the unit circle where $\bar{A}
\triangleq A + BK$. In the remainder of the paper, $u = K x$ will be
referred to as the auxiliary controller. We then impose the following
condition on the constraint sets.

\begin{assumption}\textbf{(Constraint sets)} The sets $X$ and~$U$ are convex and $K x \in U$ for $x\in  X$.\oprocend\label{asm5}
\end{assumption}

\subsection{The closed-loop system with the replay attacker}

System~\eqref{e10} together with the sensor and the actuator are
spatially separated from the operator. These entities are connected
through communication channels. In the network, there is a replay
attacker who maliciously repeats the messages delivered from the
operator to the actuator. In particular, the adversary is associated
with a memory whose state is denoted by $M^a(k)$. If a replay attack
is launched at time~$k$, the adversary executes the following: $(i)$
erases the data sent from the operator; $(ii)$ sends previous data
stored in her memory, $M^a(k)$, to the actuator; $(iii)$ maintains the
state of the memory; i.e., $M^a(k+1) = M^a(k)$. In this case, we use
$\vartheta(k)=1$ to indicate the occurrence of a replay attack. If the
attacker keeps idle at time~$k$, then data is intercepted, say $\Upsilon$,
sent from the operator to plant, and stored it in memory; i.e.,
$M^a(k+1) = \Upsilon$. In this case, $\vartheta(k) = 0$ and $u$ is
successfully received by the actuator. Without loss of any generality,
we assume that $\vartheta(-1) = \vartheta(0) = 0$.

We now define the variable $s(k)$ with initial state $s(0) = s(-1) =
0$ to indicate the consecutive number of the replay attacks. If
$\vartheta(k) = 1$, then $s(k) = s(k-1) + 1$; otherwise, $s(k) =
0$. So, the quantity $s(k)$ represents the number of consecutive
attacks up to time~$k$.

A replay attack requires spending certain amount of energy. We assume
that the energy of the adversary is limited, and adversary~$i$ is only
able to launch at most $S\geq1$ consecutive attacks. This assumption
is formalized as follows:
\begin{assumption}\textbf{(Maximum number of consecutive attacks)}
There is an integer $S\geq1$ such that $\max_{k\geq0}s(k)\leq
S$.\oprocend\label{asm3}
\end{assumption}

\begin{figure}[h]
  \centering
  \includegraphics[width=0.5\linewidth]{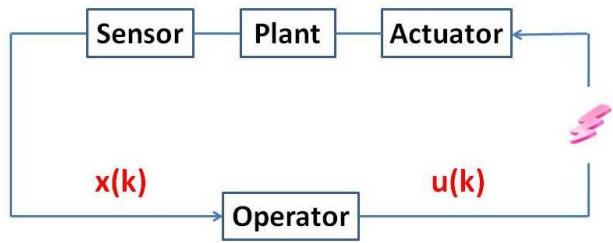}
  \caption{The closed-loop system} \label{fig_system}
\end{figure}


Replay attacks have been successfully used by the virus attack of Stuxnet~\cite{NF-LOM-EC:11,YM-TK-KB-DD-LH-AP-BS:12}. This class of attacks can be easily detected by attaching a time stamp to each control command. In the remainder of the paper, we assume that the attacks can always be detected and focus on the design and analysis of resilient controllers against them.

%

\subsection{Attack-resilient receding-horizon control law}

Here we propose a variation of the receding-horizon control in;
e.g.~\cite{DQM-JBR-CVR-POMS:00,GPL-JXM-DR-SCC:06}, to deal with the replay attacks. Our
\textbf{attack-resilient receding-horizon control law}, (for short,
AR-RHC) is stated in Algorithm~\ref{ta:algo}. In particular, at each time instant, the plant stores the whole control sequence which will be used in response to future attacks. The terminal state cost is chosen to coincide with the running state cost. This is instrumental for the analysis of performance degradation in Theorem~\ref{the1}.

\begin{algorithm}[htbp]
  \caption{The attack-resilient receding-horizon
  control law} \label{ta:algo}

\begin{algorithmic}
  \REQUIRE The following steps are first performed by the operator:
\end{algorithmic}
  \begin{algorithmic}[1]

  \STATE Choose $K$ so that $\sigma(\bar{A})$ is strictly inside the
  unit circle.

  \STATE Choose $\bar{Q} = \bar{Q}^T>0$ and obtain $\bar{P}$ by
  solving the following Lyapunov equation: \begin{align}\bar{A}^T
    \bar{P}\bar{A} - \bar{P} = - \bar{Q}.\label{Ly-eq}\end{align}

  \STATE Choose a constant $c>0$ such that
  $X_0\triangleq\{x\in\real^n\;|\;\|x\|^2_{\bar{P}}\leq c\}\subseteq
  X$.
\end{algorithmic}

\begin{algorithmic}
  \ENSURE At each $k \ge 0$, the operator, actuator and sensor execute
  the following steps:
\end{algorithmic}
\begin{algorithmic}[1]
    \STATE The operator solves the following $N$-horizon quadratic
    program, namely $N$-QP, parameterized by $x(k)\in
    X$: \begin{align*} \min_{\inputseq(k) \in \real^{m\times
          N}}&\sum_{\tau=0}^{N-1}
      \big(\|x(k+\tau|k)\|^2_P+\|u(k+\tau|k)\|^2_Q\big) +
      \|x(k+N|k)\|^2_P,\\ {\rm{s.t.}}& \quad x(k+\tau+1|k) = A x(k+\tau|k)
      + B u(k+\tau|k),\nnum\\ & \quad x(k|k) = x(k),\quad
      x(k+\tau+1|k)\in X_0, \quad u(k+\tau|k)\in U, \quad 0 \le \tau
      \le N-1,
\end{align*}
obtains the solution $\inputseq(k) \triangleq [u(k|k), \cdots,
  u(k+N-1|k)]$, and sends it to the actuator.


  \STATE If $s(k) = 0$, the actuator receives $\inputseq(k)$, sets $M^p(k+1) = \inputseq(k)$, implements $u(k|k)$, and the sensor sends $x(k+1)$ to the
  operator. If $s(k) \geq 1$, the actuator implements $u(k|k-s(k))$ in
  $M^p(k)$, sets $M^p(k+1) = M^p(k)$, and the sensor sends $x(k+1)$ to
  the operator.

  \STATE Repeat for $k = k+1$.
\end{algorithmic}
\end{algorithm}


\begin{table}[ht]
\caption{Main notations used in the following sections}
\centering
\begin{tabular}{|c|l|}
  \hline
  $\subscr{\lambda}{max}(R)$ (resp. $\subscr{\lambda}{min}(R)$) & $
\begin{array}{l}
  \text{the maximum (resp. minimum) eigenvalue of matrix $R$}
\end{array}
$
\tabularnewline
  \hline
  $\lambda \triangleq 1 - \displaystyle{\frac{\subscr{\lambda}{max}(\bar{Q})}{\subscr{\lambda}{min}(\bar{P})}}$ & $
\begin{array}{l}
  \text{positive constant, $\lambda \in (0,1)$, see~\cite{TM-NF-MK:82},}\\ \text{defined with $\bar{Q}$,
    $\bar{P}$ introduced in AR-RHC}
\end{array}
$
\tabularnewline
\hline
$\phi_N \triangleq
\displaystyle{ \frac{\subscr{\lambda}{max}(\bar{P})\subscr{\lambda}{max}(P +
    K^TQK)}{\subscr{\lambda}{min}(\bar{P})}\frac{(1-\lambda^{N+1})}{1-\lambda}}$  & $
\begin{array}{l}
\text{positive constant defined for all $N>0$,}\\\text{with $\bar{Q}$,
    $\bar{P}$, and $K$ introduced in AR-RHC}
\end{array}
$
\tabularnewline
\hline
$  \phi_{\infty} \triangleq \displaystyle{
  \frac{\subscr{\lambda}{max}(\bar{P})\subscr{\lambda}{max}(P +
    K^TQK)}{\subscr{\lambda}{min}(\bar{P})(1 - \lambda)}}$ &
$
\begin{array}{l}
  \text{positive constant defined with $\bar{Q}$,
    $\bar{P}$, and $K$ }\\\text{introduced in AR-RHC}
\end{array}
$
\tabularnewline
\hline
$\alpha_N \triangleq \displaystyle{\frac{\subscr{\lambda}{max}(K^TQK +
    \bar{A}^TP\bar{A})}{\subscr{\lambda}{min}(P)}
  \times\prod_{\kappa=0}^{N-1}(1-\frac{\subscr{\lambda}{min}(P)}{\phi_{\kappa+1}})} $ &
$
\begin{array}{l}
  \text{positive constant defined for all $N>0$,}\\\text{with $\bar{A}$ and
    $K$ introduced in AR-RHC, and $\lambda$ introduced here}
\end{array}
$
\tabularnewline
\hline
  $\rho_N \triangleq(1+\alpha_{N-1})
    (1 - \frac{\subscr{\lambda}{min}(P)}{\phi_N})$ & $\begin{array}{l}\text{a discount factor}\end{array}$
\tabularnewline
\hline
  $W(x) \triangleq \|x\|^2_{\bar{P}}$ & $
\begin{array}{l}
  \text{matrix $\bar{P}$ is the solution to Lyapunov equation~\eqref{Ly-eq}}
\end{array}
$
\tabularnewline
\hline
  $V_N$ & $
\begin{array}{l}
  \text{the optimal value function of $N$-QP}
\end{array}
$
\tabularnewline
\hline
\end{tabular}\label{ta:thms}
\end{table}

In what follows, we present the results characterizing the stability
and infinite-horizon cost induced by AR-RHC. See
Table~\ref{ta:thms}, for the main notations employed, and Section~\ref{sec:analysis} for the complete proof. Notice that the following property holds:
\begin{align*}
  \frac{\subscr{\lambda}{min}(P)}{\phi_{N}} =
  \frac{\subscr{\lambda}{min}(P)}{\subscr{\lambda}{max}(P+K^TQK)}
  \frac{\subscr{\lambda}{min}(\bar{P})}{\subscr{\lambda}{max}(\bar{P})}\frac{(1-\lambda)}{(1-\lambda^{N+1})}
  < 1.
\end{align*}
where $\lambda$ and $\phi_N$ are defined in Table~\ref{ta:thms}. On the
other hand, for $\alpha_N$ in Table~\ref{ta:thms}, $\alpha_N\searrow0$ as $N\nearrow+\infty$, and $\phi_N$ is strictly increasing in $N$ and upper bounded by $\phi_{\infty}$. Then, given any integer $S\geq1$, there is a smallest integer $N^*(S)\geq S$ such that for all $N\geq N^*(S)$, it holds that: \begin{align*}\gamma_{N,S}\triangleq (1 - \frac{\subscr{\lambda}{min}(P)}{\phi_{\infty}})
\max\{(1+\alpha_{N-S-1}),(1+\alpha_{N-1})
\prod_{\ell=N-S}^{N-1}(1+\alpha_{\ell})\} < 1.\end{align*} Analogously, given any integer $S\geq1$, there is a smallest integer $\hat{N}^*(S)\geq S$ such that for all $N\geq \hat{N}^*(S)$, it holds that \begin{align*}\hat{\gamma}_{N,S}&\triangleq (1 - \frac{\subscr{\lambda}{min}(P)}{\phi_{\infty}})^2(1+\alpha_{N-1})(1+\alpha_{N-2})\\
&\times\big(\max_{s\in\{1,\cdots,S\}}\prod_{\ell=2}^{s}(1 - \frac{\subscr{\lambda}{min}(P)}{\phi_{\infty}})(1+\alpha_{N-\ell-1})\big)
\prod_{\ell=N-S}^{N-1}(1+\alpha_{\ell}) < 1.\end{align*} One can easily verify $\hat{N}^*(S) \leq N^*(S)$. The following theorem characterizes the stability
and infinite-horizon cost of system~\eqref{e10} under AR-RHC where $V_{\ell}(x)$ represents the value of the $\ell$-QP parameterized by $x\in X$.


\begin{theorem}\textbf{(Stability and infinite-horizon cost)} Let Assumptions~\ref{asm2},~\ref{asm5} and~\ref{asm3} hold.
    \begin{enumerate}
    \item \textbf{(Exponential stability)} Suppose $N\geq \max\{N^*(S)+1,S+1\}$. Then system~\eqref{e10} under AR-RHC is exponentially stable when starting from $X_0$ with a rate of $\gamma_{N,S}$ in the sense that $V_{N-s(k-1)}(x(k))\leq \gamma_{N,S}^kV_N(x(0))$. In addition, the infinite-horizon cost of system~\eqref{e10} under AR-RHC is bounded above by $\frac{1}{1-\gamma_{N,S}}V_N(x(0))$.
  \item \textbf{(Asymptotic stability)} If $N\geq \max\{\hat{N}^*(S)+1,S+1\}$, then system~\eqref{e10} under AR-RHC is asymptotically stable when starting from $X_0$.
    \end{enumerate} \label{the1}
\end{theorem}

\begin{remark} AR-RHC with Theorem~\ref{the1} can be readily extended to several scenarios, including DoS attacks, measurement attacks and the combinations of such attacks. If the adversary launches a DoS attack on control commands, the actuator receives nothing and then performs Step 3 in AR-RHC. The adversary may produce the replay attacks on the measurements sent from the sensor to the operator. If this happens, then the operator does not send anything to the actuator and the actuator performs Step 3 in AR-RHC.\oprocend
\end{remark}

\section{Discussion and simulations}\label{sec:discussion}

\subsection{Extensions}

AR-RHC with Theorem~\ref{the1} can be readily extended to several scenarios, including DoS attacks, measurement attacks and the combinations of such attacks. If the adversary launches a DoS attack on control commands, the actuator receives nothing and then performs Step 3 in AR-RHC. The adversary may produce the replay attacks on the measurements sent from the sensor to the operator. If this happens, then the operator does not send anything to the actuator and the actuator performs Step 3 in AR-RHC.

\subsection{Explicit upper bounds on $N^*(S)$ and $\hat{N}^*(S)$}

Consider $S\geq2$ and let $\chi \triangleq (1 -
\frac{\subscr{\lambda}{min}(P)}{\phi_{\infty}})$ and $\psi \triangleq
\frac{\subscr{\lambda}{max}(K^TQK
  +\bar{A}^TP\bar{A})}{\subscr{\lambda}{min}(P)}$. Note that
\begin{align}
\gamma_{N,S} &\leq (1 - \frac{\subscr{\lambda}{min}(P)}{\phi_{\infty}})
(1+\alpha_{N-1})\prod_{\ell=N-S-1}^{N-1}(1+\alpha_{\ell})\nnum\\
&\leq \chi(1+\alpha_{N-S-1})^{S+2} \leq \beta_{N,S}
\triangleq \chi(1+\psi\chi^{N-S-1})^{S+2}.\label{e12}
\end{align}
So it suffices to find $N$ such that $\beta_{N,S} < 1$. The relation
$\beta_{N,S} < 1$ is equivalent to the following:
\begin{align*}
N-S-1 > \frac{\ln\big(\frac{1}{\psi}(\chi^{-\frac{1}{S+2}}-1)\big)}{\ln
  \chi} = \frac{\ln(\chi^{-\frac{1}{S+2}}-1)-\ln\psi}{\ln
  \chi}.
\end{align*}

Hence, an explicit upper bound on $N^*(S)$ is $\Pi_E(S) \triangleq S + 1 +
\frac{\ln(\chi^{-\frac{1}{S+2}}-1)-\ln\psi}{\ln \chi}$.

We now move to find an explicit upper bound on $\hat{N}^*(S)$. Note that \begin{align*}\hat{\gamma}_{N,S}&\leq (1 - \frac{\subscr{\lambda}{min}(P)}{\phi_{\infty}})^2(1+\alpha_{N-1})(1+\alpha_{N-2})
\big(\max_{s\in\{1,\cdots,S\}}\prod_{\ell=2}^{s}(1 - \frac{\subscr{\lambda}{min}(P)}{\phi_{\infty}})(1+\alpha_{N-\ell-1})\big)
\prod_{\ell=N-S}^{N-1}(1+\alpha_{\ell})\\
&\leq (1 - \frac{\subscr{\lambda}{min}(P)}{\phi_{\infty}})^{S+1}(1+\alpha_{N-1})(1+\alpha_{N-2})
(1+\alpha_{N-S-1})^{S-1}\prod_{\ell=N-S}^{N-1}(1+\alpha_{\ell})\\
&\leq (1 - \frac{\subscr{\lambda}{min}(P)}{\phi_{\infty}})^{S+1}(1+\alpha_{N-S-1})^{2S+1} = \chi^{S+1}(1+\psi\chi^{N-S-1})^{2S+1}.\end{align*}

So, an explicit upper bound on $\hat{N}^*(S)$ is $\Pi_A(S) \triangleq S + 1 +
\frac{\ln(\chi^{-\frac{S+1}{2S+1}}-1)-\ln\psi}{\ln \chi}$. This pair of
upper bounds clearly demonstrate that a higher computational
complexity; i.e., a larger $N$, is caused by a larger $S$, indicating
that the adversary is less energy constrained. On the other hand, the
second term in $\Pi_A(S)$ approaches a constant as $S$ goes to
infinity. So $\Pi_A(S)$ can be upper bounded by an affine
function. However, the second term in $\Pi_E(S)$ dominates when $S$ is
large. That is,  exponential stability demands a much
higher cost than asymptotic stability when $S$ is large.

\subsection{A reverse scenario}

Reciprocally, for any horizon $N\geq1$, there is a largest integer
$S^*(N) \leq N-1$ (resp. $\hat{S}^*(N) \leq N-1$) such that for all
$S\leq S^*(N)$ (resp. $S\leq \hat{S}^*(N)$), it holds that
$\gamma_{N,S} < 1$ (resp. $\hat{\gamma}_{N,S} <
1$). Theorem~\ref{the1} still applies to this reverse scenario and
characterizes the ``security level'' or ``amount of resilience'' that
the proposed receding-horizon control algorithm possesses.

\subsection{Optimal resilience management}

The analysis of Theorem~\ref{the1} quantifies the cost and constraints
that allow the AR-RHC algorithm to work despite consecutive attacks
under limited computation capabilities. These metrics can be used for optimal resilience management of a network as follows.

As~\cite{SA-GAS-SSS:10}, we consider a set of players $V\triangleq
\{1,\cdots,N\}$ where the players share a communication network and each of them is associated with a decoupled dynamic system:
\begin{align}
x_i(k+1) = A_i x_i(k) + B_i u_i(k).\label{e11}
\end{align}
Each player~$i$ implements his own AR-RHC with horizon
$N_i$. The notations in the previous sections can be defined analogously for each player and the set of the notations of player~$i$ will be indexed by~$i$.

By~\eqref{e12}, we associate player~$i$ with the following cost
function:
\begin{align}
\mathcal{C}_i(M) =
\big(1+\psi_i\chi_i^{N_i-\mathcal{S}(\textbf{1}^TM)}\big)^{\mathcal{S}(\textbf{1}^TM)+1}
+\frac{1}{2}a_iM_i^2,\label{e20}\end{align}
where $M_i\in[M_{i,\min},M_{i,\max}]\subset\real_{>0}$ is the security
investment of player~$i$, $a_i\in\real_{>0}$ is a weight on the security
cost the and $\textbf{1}$ is the vector with $N$ ones. The non-negative real value
$\mathcal{S}(\textbf{1}^TM)$ represents the security level given the
investment vector $M$ of all players, where $\mathcal{S} :
\real_{\geq0}\rightarrow\real_{\geq0}$ is convex, non-decreasing, and
smooth. We assume that each player has a fixed computational power,
and so $N_i$ is fixed. The players need to make the investment such
that
\begin{align}
\mathcal{S}(\textbf{1}^TM)\leq
  \min_{i\in V}S_i^*(N_i).\label{e13}
\end{align}

\begin{remark} Note that $S$ is an integer in~\eqref{e12}. In~\eqref{e20} and~\eqref{e13},
we use the real value of $\mathcal{S}(\textbf{1}^TM)$ as an approximation.\label{rem2}\oprocend
\end{remark}

We now compute the first-order partial derivative
of $\mathcal{C}_i$ as follows:
\begin{align*}
\frac{\partial \mathcal{C}_i}{\partial M_i} &=
-\ln(1+\psi_i\chi_i^{N_i-\mathcal{S}(\textbf{1}^TM)})
(1+\psi_i\chi_i^{N_i-\mathcal{S}(\textbf{1}^TM)})^{\mathcal{S}(\textbf{1}^TM)+1}(\ln\chi_i)\psi_i
\chi_i^{N_i-\mathcal{S}_i(M)}(\frac{\partial\mathcal{S}}{\partial
  y})^2 + a_iM_i\end{align*} where we use the shorthand $y\triangleq
\textbf{1}^TM$. With this, we further
derive the second-order partial derivative as
follows: \begin{align*}\frac{\partial^2 \mathcal{C}_i}{\partial M_i^2}
  &= \psi_i^2(\ln\chi_i)^2\chi_i^{2(N_i-\mathcal{S}_i(\textbf{1}^TM))}
  (1+\psi_i\chi_i^{N_i-\mathcal{S}(\textbf{1}^TM)})^{\mathcal{S}(\textbf{1}^TM)}
  (\frac{\partial\mathcal{S}}{\partial
    y})^3 + a_i\\ &+(\ln(1+\psi_i\chi_i^{N_i-\mathcal{S}(\textbf{1}^TM)}))^2
  (1+\psi_i\chi_i^{N_i-\mathcal{S}(\textbf{1}^TM)})^{\mathcal{S}(\textbf{1}^TM)+1}
  (\psi_i\ln\chi_i
  \chi_i^{N_i-\mathcal{S}_i(\textbf{1}^TM)})^2(\frac{\partial
    \mathcal{S}}{\partial
    y})^4\\ &+\ln(1+\psi_i\chi_i^{N_i-\mathcal{S}(\textbf{1}^TM)})
  (1+\psi_i\chi_i^{N_i-\mathcal{S}(\textbf{1}^TM)})^{\mathcal{S}(\textbf{1}^TM)+1}
  \psi_i(\ln\chi_i)^2
  \chi_i^{N_i-\mathcal{S}_i(\textbf{1}^TM)}(\frac{\partial
    \mathcal{S}}{\partial
    y})^3\\ &+2(\ln(1+\psi_i\chi_i^{N_i-\mathcal{S}(\textbf{1}^TM)}))^2
  (1+\psi_i\chi_i^{N_i-\mathcal{S}(\textbf{1}^TM)})^{\mathcal{S}(\textbf{1}^TM)+1}
  \psi_i(-\ln\chi_i)
  \chi_i^{N_i-\mathcal{S}_i(\textbf{1}^TM)}\frac{\partial
    \mathcal{S}}{\partial y}\frac{\partial^2 \mathcal{S}}{\partial
    y^2}.\end{align*}

Recall that $\chi_i\in(0,1]$ and $\mathcal{S}$ is non-decreasing and
  convex. So $\frac{\partial^2 \mathcal{C}_i}{\partial M_i^2} \geq
  0$ and $\mathcal{C}_i$ is convex in $M_i$. Analogously, one can show that $\mathcal{C}_i$ is convex in $M$.

\subsubsection{Competitive resource allocation scenario}

Consider a \emph{resilience management game}, where each player~$i$
minimizes his cost $\mathcal{C}_i(M)$, subject to the common
constraint~\eqref{e13} and his private constraint
$M_i\in[M_{i,\min},M_{i,\max}]\subset\real_{>0}$. Since
$\mathcal{C}_i$ and $\mathcal{S}$ are convex in $M_i$, then the game is a
generalized convex game. The distributed algorithms in~\cite{MZ-EF:12}
can be directly utilized to numerically compute a Nash equilibrium of the
resilience management game, and the algorithms in~\cite{MZ-EF:12} are able to tolerate transmission delays and packet dropouts.

\begin{remark} The paper~\cite{SA-GAS-SSS:10} considers a set of identical and independent networked control systems and each of them aims to solve an infinite-horizon LQG
problem. The authors study a different security game where the decisions of each player are binary, participating in the security investment or not.\oprocend\label{rem1}
\end{remark}

\subsubsection{Cooperative resource allocation scenario}

Consider a \emph{resilience management optimization problem}, where
the players aim to collectively minimize $\sum_{i\in V}\mathcal{C}_i(M)$,
subject to the global constraint~\eqref{e13} and the private
constraint $M_i\in[M_{i,\min},M_{i,\max}]\subset\real_{>0}$. Since
$\mathcal{C}_i$ and $\mathcal{S}$ are convex, then the problem is a
convex program. The distributed algorithms in~\cite{MZ-SM:09c} can be
directly exploited to numerically compute a global minimizer of this problem, and the algorithms in~\cite{MZ-SM:09c} are robust to the dynamic changes of inter-player topologies.

\subsection{Simulations}

In this section, we provide a numerical example to illustrate the performance of our algorithm. The set of system parameters are given as follows:
\begin{align*}
&A = \left[\begin{array}{cc}
      2 & 1 \\
      1 & 2
    \end{array}\right],\quad B = \left[\begin{array}{c}
      2 \\
      1
    \end{array}\right],\quad K = [-3.25\;\;-3],\quad P = I, \quad Q = 1,\\
&\bar{Q} = I,\quad \bar{P} = \left[\begin{array}{cc}
      25.6667 & 13.3333 \\
      13.3333 & 8.2963
    \end{array}\right],\quad c = 100,\quad u_{\max} = 500.
\end{align*}

Figure~\ref{fig_simulation} shows the temporal evolution of
$\|x(k)\|^2$ under three attacking horizons $S = 0, 2, 5$. One can see
that a larger $S$ induces a longer time to converge, and larger
oscillation before reaching the equilibrium. In our simulations, a
smaller horizon $N=15$ than the one determined theoretically is
already sufficient to achieve system stabilization.

\begin{figure}[h]
  \centering
  \includegraphics[width=0.5\linewidth]{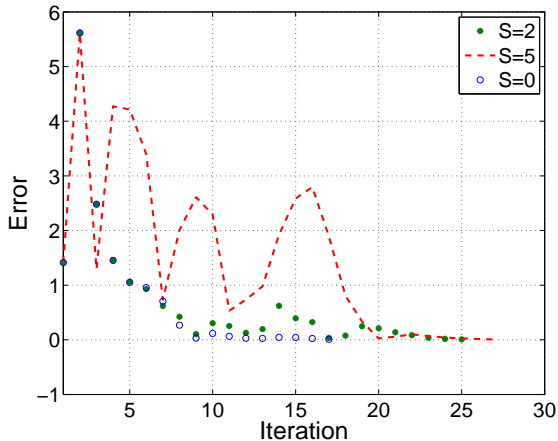}
  \caption{The trajectories of $\|x(k)\|^2$ under the attack-resilient
    receding-horizon control algorithm} \label{fig_simulation}
\end{figure}

\section{Conclusions}

In this paper, we have studied a resilient control problem where a
linear dynamic system is subject to the replay and DoS attacks. We
have proposed a variation of the receding-horizon control law for the
operator and analyzed system stability and performance
degradation. We have also studied a class of competitive
  (resp. cooperative) resource allocation problems for resilient
  networked control systems. Extension to multi-agent systems will be
  considered in the future.

\section{Appendix: Technical proofs}\label{sec:analysis}


The proofs toward Theorem~\ref{the1} are collected in this section. In
particular, the proofs for the intermediate lemmas are based on the
corresponding results in our previous paper~\cite{MZ-SM:ACC11} on
deception attacks. The proofs for the main theorem are new and not
included in~\cite{MZ-SM:ACC11}. In the proof of Theorem~\ref{the1}, we
choose $V_{N-s(k-1)}(x(k))$ as a Lyapunov function candidate. To
analyze its convergence, we first establish several instrumental
properties of $V_N$, including  monotonicity, diminishing rations
with respect to $N$ and decreasing property.

Recall the definitions of $\lambda$, $\alpha_N$, $\phi_N$, and
$\phi_{\infty}$ summarized in Table~\ref{ta:thms}. It follows
from~\cite{TM-NF-MK:82} that $\lambda\in(0,1)$, and clearly, $1 \le
\phi_N\leq\phi_{\infty}$ for any $N\in {\mathbb{Z}}_{>0}$. Observe
that the following holds for any
$\kappa\in{\mathbb{Z}}_{>0}$: \begin{align*}
  \frac{\subscr{\lambda}{min}(P)}{\phi_{\kappa+1}} =
  \frac{\subscr{\lambda}{min}(P)}{\subscr{\lambda}{max}(P+K^TQK)}
  \frac{\subscr{\lambda}{min}(\bar{P})}{\subscr{\lambda}{max}(\bar{P})}
  \frac{1-\lambda}{1-\lambda^{{\kappa}+2}} \geq
  \frac{\subscr{\lambda}{min}(P)}{\subscr{\lambda}{max}(P+K^TQK)}
  \frac{\subscr{\lambda}{min}(\bar{P})}{\subscr{\lambda}{max}(\bar{P})}(1-\lambda)\in(0,1).
\end{align*} This ensures the monotonicity of $\alpha_N$ and,
moreover, that $\alpha_N\searrow0$ as $N\nearrow+\infty$.

We show the forward invariance property of system~\eqref{e10} in $X_0$
under $Kx$.

\begin{lemma}[Forward invariance in $X_0$]
  The set $X_0$ is forward invariant for system~\eqref{e10} under the auxiliary controller $Kx$ with the control constraint $U$; i.e., for any $x \in X_0$, it holds that $u = Kx \in U$ and
  $\bar{A}x \in X_0$.
\label{lem2}
\end{lemma}
\begin{proof}The differences of $W$ along the
  trajectories of the dynamics~\eqref{e10} under $u(k) =
  Kx(k)$, $x(k) = x$ can be characterized by:
  \begin{align}
    & W(x(k+1)) - W(x) = \|\bar{A}x(k+1)\|^2_{\bar{P}} -
    \|x(k)\|^2_{\bar{P}} = - \|x\|^2_{\bar{Q}} \leq
    -\subscr{\lambda}{min}(\bar{Q})\|x\|^2,\label{e42}
\end{align}
where $W(x)$, $\bar{A}$, $\bar{P}$ and $\bar{Q}$ are given in Table~\ref{ta:thms}, and in the second equality we apply the Lyapunov equation~\eqref{Ly-eq}. Since
$\bar{Q} > 0$, then $W(x(k+1)) \leq W(x)$. Since $x$ belongs to
$X_0$, so does $x(k+1)$. Since $X_0\subseteq X$, we know
that $u(k)\in U$ by Assumption~\ref{asm5}. The forward invariance
property of $X_0$ for system~\eqref{e10} follows.
\end{proof}

On the other hand, one can see that the $N$-QP parameterized by $x\in X_0$ has at
least one solution generated by the auxiliary controller.

\begin{lemma}[Feasibility of the $N$-QP] For any
  $x \in X_0$, consider system~\eqref{e10} with $x(k|k) = x$
  and $u(k+\tau|k) = Kx(k+\tau|k)$, for $0\leq \tau \leq N-1$. Then,
  ${\inputseq}(k)$ is a feasible solution to the $N$-QP parameterized by $x(k)\in X_0$.\label{lem3}
\end{lemma}

\begin{proof} It is a direct result of Lemma~\ref{lem2} and Assumption~\ref{asm5}.
\end{proof}

The following lemma demonstrates that $V_N$ is bounded above and below by two quadratic functions, respectively.

\begin{lemma}\textbf{(Positive-definite and decrescent properties of $V_N$) }The function
  $V_N$ is quadratically bounded above and below as
  $\subscr{\lambda}{min}(P)\|x\|^2\leq V_N(x) \leq
  \phi_N\|x\|^2$ for any $x\in X_0$. \label{lem4}
\end{lemma}

\begin{proof}Consider any $x \in X_0$. It is easy to see that
  $V_N(x)\geq\subscr{\lambda}{min}(P)\|x\|^2$, and thus
  positive definiteness of $V_N$ follows. We now proceed to show
  that $V_N$ is decrescent. In order to simplify the notations
  in the proof, we will drop the dependency on time~$k$ in what
  follows. Toward this end, we let $\{x(\tau)\}_{\tau\geq0}$ be the
  solution produced by the system $x(\tau+1) = \bar{A} x(\tau)$, that
  is, the closed-loop system solution of the dynamics~\eqref{e10} under the auxiliary controller $Kx$, with initial state $x(0) = x\in X_0$. We denote $x(\tau|0 )\equiv
  x(\tau)$ and $u(\tau|0) \equiv u(\tau)$. Recall the
  estimate~\eqref{e42}:
\begin{align}
  W(x(\tau+1)) &\leq W(x(\tau))-
  \subscr{\lambda}{min}(\bar{Q})\|x(\tau)\|^2 \leq W(x(\tau))-
  \frac{\subscr{\lambda}{max}(\bar{Q})}{\subscr{\lambda}{max}(\bar{P})}
  W(x(\tau)),\label{e45}
\end{align}
where we use the property that
$\subscr{\lambda}{min}(\bar{P})\|x\|^2\leq W(x) \leq
\subscr{\lambda}{max}(\bar{P})\|x\|^2$. It follows from
Lemma~\ref{lem3} that the sequence of control commands $u(\tau) =
Kx(\tau)$ for $0\leq \tau \leq N-1$ consists of a feasible solution
to the $N$-QP parameterized by $x\in X_0$. Then
we achieve the following on $V_N(x)$:
\begin{align}
  V_N(x)
  &\leq \sum_{\tau=0}^{N-1}\big(\|x(\tau)\|^2_P+\|Kx(\tau)\|^2_Q\big) + \|x(N)\|^2_P\nnum\\
  &\leq \sum_{\tau=0}^{N-1}\subscr{\lambda}{max}(P + K^TQK)\|x(\tau)\|^2+ \subscr{\lambda}{max}(P)\|x(N)\|^2\nnum\\
  &\leq \frac{\subscr{\lambda}{max}(P + K^TQK)}{\subscr{\lambda}{min}(\bar{P})}\sum_{\tau=0}^{N-1}W(x(\tau))+
  \frac{\subscr{\lambda}{max}(P)}{\subscr{\lambda}{min}(\bar{P})}W(x(N)).\label{e46}
\end{align}
Substituting inequality~\eqref{e45} into~\eqref{e46}, we obtain the
following estimates on $V_N(x)$:
\begin{align}
  V_N(x)
  &\leq \frac{\subscr{\lambda}{max}(P + K^TQK)}{\subscr{\lambda}{min}(\bar{P})}W(x)\sum_{\tau=0}^{N-1}\lambda^{\tau}
  + \frac{\subscr{\lambda}{max}(P)}{\subscr{\lambda}{min}(\bar{P})}W(x)\lambda^N\nnum\\
  &\leq\frac{\subscr{\lambda}{max}(\bar{P})\subscr{\lambda}{max}(P +
    K^TQK)}{\subscr{\lambda}{min}(\bar{P})}\frac{1-\lambda^{N+1}}{1-\lambda}\|x\|^2.\nnum
\end{align}
where we use the fact $\lambda = 1-\frac{\subscr{\lambda}{max}(\bar{Q})}
{\subscr{\lambda}{max}(\bar{P})}\in(0,1)$ in~\cite{TM-NF-MK:82}.
The decrescent property of $V_N$ immediately follows from the above relations.
\end{proof}

Next, one can show that for any $x\in X_0$, $V_N(x)$ does not decrease as $N$ increases.
\begin{lemma}[Monotonicity of $V_N$]
  The optimal value function $V_N$ is monotonic in $N$; i.e.,
  for any $x\in X_0$, $V_{N'}(x)\leq V_N(x)$ for $N'
  < N$.\label{lem5}
\end{lemma}

\begin{proof}
  Consider $N' < N$, and denote by $J_N$ and $J_{N'}$ the objective functions of the $N$-QP and the $N'$-QP, respectively. Let ${\inputseq}_N$ be a solution to the
  $N$-QP parameterized by $x$, with $\inputseq_N =
  [u(0),\dots,u(N-1)]$, and let ${\inputseq}_{N'}$, with $\inputseq_{N'} =
  [u(0),\dots,u(N'-1)]$, be a solution to the
  ${N'}$-QP parameterized by $x\in X_0$. We construct $\tilde{{\inputseq}}_{N'} \in U^{N'}$, a truncated version of
  ${\inputseq}_N$, in such a way that $\tilde{u}(k) = u(k)$ for
  $0\leq k \leq {N'}-1$. Since ${\inputseq}_N$ is a
  solution to the $N$-QP parameterized by $x$, then one can show
  that $\tilde{\inputseq}_{N'}$
  is a feasible solution to the ${N'}$-QP parameterized by $x$. This
  renders the following upper bound on $V_{N'}(x)$:
\begin{align}
  V_{N'}(x) = J_{N'}(x,{\inputseq}_{N'}) \leq
  J_{N'}(x,\tilde{\inputseq}_{N'}).
\label{e48}
\end{align}
Denote by $\stateseq_N \triangleq [x(0),\cdots,x(N)]$ the
corresponding trajectory to $\inputseq_N$ with initial
state $x(0) = x$ and by $\tilde{\stateseq}_{N'} \triangleq
[\tilde{x}(0),\cdots,\tilde{x}_{N'}]$ the corresponding trajectory
generated by the sequence of $\tilde{{\inputseq}}_{N'}$ with the initial state $\tilde{x}(0) = x$. Since $\tilde{{\inputseq}}_{N'}$ is a truncated version of ${\inputseq}_N$,
we have that $\tilde{x}(k) = x(k)$ for $0\leq k \leq N'$. Denote
further $\tilde{\inputseq}_{N'} \triangleq
[\tilde{u}(0),\cdots,\tilde{u}(N'-1)]$. Then we have
\begin{align}
  &J_{N'}(x,\tilde{{\inputseq}}_{N'})
  =\sum_{k=1}^{{N'}}
  \big(\|\tilde{x}(k)\|^2_P+\|\tilde{u}(k)\|^2_{Q}\big) + \|\tilde{x}(N')\|^2_P\nnum\\
  &=\sum_{k=1}^{{N'}}
  \big(\|x(k)\|^2_P+\|u(k)\|^2_{Q}\big) + \|x(N')\|^2_P
  \leq\sum_{k=1}^{N}
  \big(\|x(k)\|^2_P+\|u(k)\|^2_{Q}\big) + \|x(N)\|^2_P = V_N(x).\nnum
\end{align}

The combination of~\eqref{e48} and the above relation establishes that
$V_{N'}(x)\leq V_N(x)$ for $x\in X_0$.
\end{proof}

The following lemma formalizes that for any $x\in X_0$, the
difference between $V_{N+1}(x)$ and $V_N(x)$ decreases as
$N$ increases by noting that $V_N(x)\leq V_{N+1}(x)$ and
$\alpha_N$ is strictly decreasing in $N$, where $V_{N+1}$ and
$V_N$ are the optimal value functions for the
$(N+1)$-QP and the $N$-QP, respectively. This property
is referred to as the property of diminishing ratios of $V_N$ in
$N$ by noting that $\alpha_N\searrow0$ as $N\nearrow+\infty$.

\begin{lemma}\textbf{(The diminishing ratios of $V_N$ in $N$)} The optimal
  value function $V_N$ is diminishingly increasing in $N$ in
  such a fashion that
  $\frac{V_{N+1}(x)-V_N(x)}{V_N(x)}\leq \alpha_N$
  for any $x\in X_0$.\label{lem6}
\end{lemma}

\begin{proof} Let ${\inputseq}_N$, with
  $\inputseq_N = [u(0),\dots,u(N-1)]$, be a solution to the $N$-QP parameterized by $x \in X_0$.  Let $\stateseq_N=[x(0),\dots,x(N)]$, $x(0) = x$, be the corresponding
  trajectory. Notice that $x(k)\in X_0$ for $0\leq k\leq N$. We
  construct an extended version $\tilde{\inputseq}_{N+1} \in U^{N+1}$
  of ${\inputseq}_N$ as $\tilde{\inputseq}_{N+1} = [u(0),\dots,u(N-1),
  K x(N)]$. Since $x(N)\in X_0$, then $\tilde{x}(N+1) :=
  \bar{A}x(N) \in X_0$ by Lemma~\ref{lem2}, implying that
  $\tilde{\inputseq}_{N+1}$ consists of a feasible solution to the $(N+1)$-QP
  parameterized by $x$.  Then we establish the following upper bounds
  on $V_{N+1}(x)$:
  \begin{align}
    V_{N+1}(x)\leq J_{N+1}(x,\tilde{\inputseq}_{N+1})=J_{N}(x ,{\inputseq}_N) + \|Kx(N)\|^2_Q + \|\tilde{x}(N+1)\|_P^2\leq V_N(x) + \varsigma\|x(N)\|^2,\label{e60}
  \end{align}
  where $\varsigma := \subscr{\lambda}{max}(K^TQK +
  \bar{A}^TP\bar{A})$. We now turn our attention to find a relation
  between $\|x(N)\|^2$ and $V_N(x)$. To achieve this, we will
  show the following holds for $\ell\in\{0,\cdots,N\}$ by induction:
\begin{align}
  &V_{\ell}(x(N-\ell)) \leq
  \prod_{\kappa=\ell}^{N-1}(1-\frac{\subscr{\lambda}{min}(P)}{\phi_{\kappa+1}})
  V_{N}(x).\label{e51}
\end{align}
It follows from Bellman's principle of optimality that
\begin{align}
  V_{N}(x) = \|x(0)\|_P^2 + \|u(0)\|^2_Q +
  V_{N-1}(x(1)).\nnum
\end{align}
We can further see that $V_N(x) - V_{N-1}(x(1))$ is
lower bounded in the following way:
\begin{align}
  V_N(x) - V_{N-1}(x(1))
  \geq \subscr{\lambda}{min}(P)\|x\|^2
  \geq\frac{\subscr{\lambda}{min}(P)}{\phi_N}V_N(x),\label{e9}
\end{align}
where we use the decrescent property in Lemma~\ref{lem4} in the last
inequality.  Rearrange terms in~\eqref{e9} and it renders
that~\eqref{e51} holds for $\ell = N-1$.

Assume that~\eqref{e51} holds for some $\ell+1\in\{1,\cdots,N-1\}$;
i.e., the following holds:
\begin{align}
  V_{\ell+1}(x(N-\ell-1))\leq
  \prod_{\kappa=\ell+1}^{N-1}(1-\frac{\subscr{\lambda}{min}(P)}{\phi_{\kappa+1}})V_{N}(x).
\label{e53}
\end{align}
Similar to~\eqref{e9}, it follows from Bellman's principle of optimality and Lemma~\ref{lem4} that \begin{align}
  V_{\ell+1}(x(N-\ell-1)) - V_{\ell}(x(N-\ell))
  \geq \subscr{\lambda}{min}(P)\|x(N-\ell-1)\|^2
  \geq\frac{\subscr{\lambda}{min}(P)}{\phi_{\ell+1}}V_{\ell+1}(x(N-\ell-1)).\label{e54}
\end{align}
Combining~\eqref{e53} and~\eqref{e54} renders that
\begin{align}
  V_{\ell}(x(N-\ell))
  \leq (1-\frac{\subscr{\lambda}{min}(P)}{\phi_{\ell+1}}) V_{\ell+1}(x(N-\ell-1))
  \leq\prod_{\kappa=\ell}^{N-1}(1-\frac{\subscr{\lambda}{min}(P)}{\phi_{\kappa+1}})
  V_{N}(x).\nnum
\end{align}
This implies~\eqref{e51} holds for $\ell$. By induction, we conclude
that~\eqref{e51} holds for $\ell\in\{0,\cdots,N\}$.
Let $\ell=0$ in~\eqref{e51}, and we have that $V_0(x(N)) \leq
\prod_{\kappa=0}^{N-1}(1-\frac{\subscr{\lambda}{min}(P)}{\phi_{\kappa+1}})
V_{N}(x)$, implying that
$\|x(N)\|^2\leq\frac{1}{\subscr{\lambda}{min}(P)}
\prod_{\kappa=0}^{N-1}(1-\frac{\subscr{\lambda}{min}(P)}{\phi_{\kappa+1}})V_{N}(x)$
by Lemma~\ref{lem4}. By combining this relation with~\eqref{e60}, we
obtain the desired relation between $V_{N+1}$ and $V_N$.
\end{proof}

A relation between $V_N(x(k+1|k))$ and $V_N(x(k))$ for $x(k)\in X_0$, and $x(k+1|k)$ generated
through the $N$-QP, is found next.
\begin{lemma}[Decreasing property of $V_N$ in $X_0$] With $x(k+1|k)$ generated through the $N$-QP starting from $x(k)$, the
  following decreasing property holds for any $x(k)\in X_0$:
  \begin{align}V_N(x(k+1|k))\leq \rho_N V_N(x(k)).\nnum
  \end{align} \label{lem9}
\end{lemma}

\begin{proof}
  With Lemma~\ref{lem4} and~\ref{lem6}, we reach the following
  relation between $V_N(x(k+1|k))$ and $V_N(x(k))$ for
  any $x(k)\in X_0$:
  \begin{align}
    &V_N(x(k+1|k))\leq (1+\alpha_{N-1})V_{N-1}(x(1))\leq (1+\alpha_{N-1}) (V_N(x(k)) - \|x(k)\|_P^2)\nnum\\
    &\leq (1+\alpha_{N-1}) (V_N(x(k)) - \subscr{\lambda}{min}(P)\|x(k)\|^2)
    \leq (1+\alpha_{N-1}) (1 -
    \frac{\subscr{\lambda}{min}(P)}{\phi_N})V_N(x(k)),\nnum
  \end{align} where Lemma~\ref{lem6} and Lemma~\ref{lem4} are used in the first and last inequalities,
  respectively, by noting that $x(k+1|k)$ and $x(k)$ in
  $X_0$.
\end{proof}

\textbf{Proof of Theorem~\ref{the1}:}

\begin{proof} \textbf{[Part 1: Exponential stability]} Let us consider the first part of $N\geq \max\{N^*(S)+1,S+1\}$. Recall that $x(0)\in X_0$ and the state constraint $X_0$ is enforced in the $N$-QP. Repeatedly apply Lemma~\ref{lem3} and we have that $x(k)\in X_0$ for all $k\geq0$. We now distinguish four cases:

\emph{Case 1:} $\vartheta(k) = 1$ and $\vartheta(k-1) = 0$. For this case, $s(k) =1$, $s(k-1) = 0$, and we have \begin{align}&V_{N-s(k)}(x(k+1)) = V_{N-1}(x(k+1)) \leq \rho_{N-1} V_{N-1}(x(k))\nnum\\
&\leq \rho_{N-1} V_{N}(x(k)) = \rho_{N-1} V_{N-s(k-1)}(x(k)),\nnum\end{align} where the first inequality uses Lemma~\ref{lem9} and the principle of optimality, and the second one exploits Lemma~\ref{lem5}.

\emph{Case 2:} $\vartheta(k) = \vartheta(k-1) = 0$. Here, $s(k) = s(k-1) = 0$. By Lemma~\ref{lem9}, we have \begin{align*}V_{N-s(k)}(x(k+1)) = V_N(x(k+1)) \leq \rho_N V_N(x(k)) = \rho_N V_{N-s(k-1)}(x(k)).\end{align*}

\emph{Case 3:} $\vartheta(k) = \vartheta(k-1) = 1$. Note that $s(k) = s(k-1) + 1$, and then \begin{align*}V_{N-s(k)}(x(k+1)) \leq \rho_{N-s(k)} V_{N-s(k)}(x(k)) \leq \rho_{N-s(k)} V_{N-s(k-1)}(x(k)),\end{align*} where the first inequality utilizes Lemmas~\ref{lem9} and the principle of optimality, and the second one exploits Lemma~\ref{lem5}.

\emph{Case 4:} $\vartheta(k) = 0$ and $\vartheta(k-1) = 1$. For this case, we have $s(k) = 0$, $s(k-1)\geq1$ and thus \begin{align*}V_{N-s(k)}(x(k+1)) = V_N(x(k+1)) \leq \rho_N V_N(x(k))\leq \rho_N\prod_{\ell=N-s(k-1)}^{N-1}(1+\alpha_{\ell}) V_{N-s(k-1)}(x(k)),\end{align*} where the last inequality repeatedly applies Lemma~\ref{lem6}.

Combine the above four cases, and it renders the following:
\begin{align}V_{N-s(k)}(x(k+1)) &\leq \max\{\max_{s\in\{1,\cdots,S\}}\{\rho_{N-s}\},
\rho_N\max_{s=1,\cdots,S}\{\prod_{\ell=N-s}^{N-1}(1+\alpha_{\ell})\}\} V_{N-s(k-1)}(x(k))\nnum\\
&\leq \gamma_{N,S} V_{N-s(k-1)}(x(k)).\label{e15}
\end{align}

Since $0<\gamma_{N,S}<1$, $\{V_{N-s(k-1)}(x(k))\}$ exponentially diminishes, and the following holds: \begin{align}
    V_{N-s(k-1)}(x(k))\leq \gamma_{N,S}^kV_N(x(0)).\label{e14}
  \end{align} Recall $N \geq S+1$. It follows from~\eqref{e14} that the
  infinite-horizon cost is characterized as
  follows:
  \begin{align}
    \sum_{k=0}^{+\infty}(\|x(k)\|_P^2 + \|u(k)\|_Q^2) \leq
    \sum_{k=0}^{+\infty}V_{N-s(k-1)}(x(k))\leq
    \sum_{k=0}^{+\infty}\gamma_{N,S}^kV_N(x(0))=\frac{1}{1-\gamma_{N,S}}V_N(x(0)).\nnum
  \end{align} We then have finished the proofs for the first part.

  \textbf{[Part 2: Asymptotic stability]} We now proceed to show the second part of $N\geq \max\{\hat{N}^*(S)+1,S+1\}$. Towards this end, we partition the time horizon $\{0,1,\cdots\}$ into a sequence of subsets $\{C_1, A_1, C_2, A_2, \cdots\}$ where $C_i = \{c_i^L,\cdots,c_i^U\}$ and $A_i = \{a_i^L,\cdots,a_i^U\}$ with for $k\in C_i$, then $\vartheta(k) = 0$; and $k\in A_i$, then $\vartheta(k) = 1$. Note that $c_0^L = 0$ and $a_i^L = c_i^U + 1$.

\emph{Case 1:} $k\in C_i\setminus\{c_i^L\}$. Note that $s(k) = s(k-1) = 0$ for all $k\in C_i\setminus\{c_i^L\}$. By Lemma~\ref{lem9}, we have \begin{align*}V_{N-s(k)}(x(k+1)) \leq \rho_N V_{N-s(k-1)}(x(k)), \quad \forall k\in C_i\setminus\{c_i^L\}.\end{align*}

\emph{Case 2:} $k = a_i^L$. Note that $\vartheta(a_i^L) = 1$ and $\vartheta(a_i^L-1) = 0$. By Case 1 in Part 1, we have \begin{align*}V_{N-s(a_i^L)}(x(a_i^L+1))\leq \rho_{N-1}V_{N-s(a_i^L-1)}(x(a_i^L)).\end{align*}

\emph{Case 3:} $k = A_i\setminus\{a_i^L\}$. Recall that $\vartheta(k) = 1$ for $k\in A_i$. By repeating the result of Case 3 in Part 1, we have \begin{align*}V_{N-s(k)}(x(k+1))\leq \prod_{\ell=2}^{k-a_i^L}\rho_{N-\ell}V_{N-s(a_i^L)}(x(a_i^L+1)),\quad \forall k\in A_i\setminus\{a_i^L\}.\end{align*}


\emph{Case 4:} $k = c_{i+1}^L = a_i^U + 1$. Note that $\vartheta(c_{i+1}^L) = 0$ and $\vartheta(c_{i+1}^L-1) = 1$. By Case 4 in Part 1, it holds that \begin{align*}V_{N-s(c_{i+1}^L)}(x(c_{i+1}^L+1)) \leq \rho_N\prod_{\ell=N-s(c_{i+1}^L-1)}^{N-1}(1+\alpha_{\ell}) V_{N-s(c_{i+1}^L-1)}(x(c_{i+1}^L)).\end{align*}


The combination of the above four relations renders the following: \begin{align*}V_{N-s(c_{i+1}^L)}(x(c_{i+1}^L+1))&\leq \rho_N
  \prod_{\ell=N-s(c_{i+1}^L-1)}^{N-1}(1+\alpha_{\ell}) V_{N-s(c_{i+1}^L-1)}(x(c_{i+1}^L))\\
  &= \rho_N
  \prod_{\ell=N-s(c_{i+1}^L-1)}^{N-1}(1+\alpha_{\ell}) V_{N-s(c_{i+1}^L-1)}(x(a_i^U+1))\\
  &\leq \rho_N\prod_{\ell=2}^{a_i^U-a_i^L}\rho_{N-\ell}
  \prod_{\ell=N-s(c_{i+1}^L-1)}^{N-1}(1+\alpha_{\ell}) V_{N-s(c_{i+1}^L-1)}(x(a_i^L+1))\\
  &\leq \rho_N\rho_{N-1}\prod_{\ell=2}^{a_i^U-a_i^L}\rho_{N-\ell}
  \prod_{\ell=N-s(c_{i+1}^L-1)}^{N-1}(1+\alpha_{\ell}) V_{N-s(c_{i+1}^L-1)}(x(a_i^L))\\
  &= \rho_N\rho_{N-1}\prod_{\ell=2}^{a_i^U-a_i^L}\rho_{N-\ell}
  \prod_{\ell=N-s(c_{i+1}^L-1)}^{N-1}(1+\alpha_{\ell}) V_{N-s(c_{i+1}^L-1)}(x(c_i^U+1))\\
  &\leq \hat{\gamma}_{N,S}V_{N-s(c_{i+1}^L-1)}(x(c_i^L)),\end{align*} where the four inequalities sequentially apply Cases 4 to 1. Since $\hat{\gamma}_{N,S}\in(0,1)$, the subsequence $\{V_{N-s(c_{i+1}^L-1)}(x(c_{i+1}^L))\}$ exponentially decreases.

  By the above four cases, it is not difficult to verify that the following holds for all $k\in A_i\cup C_i\setminus\{c_i^L\}$: \begin{align*}V_{N-s(k-1)}(x(k))&\leq \max\{\rho_{N-1},1\}\rho_N\max_{s\in\{2,\cdots,S\}}\prod_{\ell=2}^s\rho_{N-\ell}
  \prod_{\ell=N-s(c_{i+1}^L-1)}^{N-1}(1+\alpha_{\ell}) V_{N-s(c_{i+1}^L-1)}(x(c_i^L)).\end{align*} Hence, the whole sequence $\{V_{N-s(k-1)}(x(k))\}$ diminishes. It establishes the asymptotical stability.

\end{proof}


\end{document}